\newif\ifarxiv\arxivtrue
\newif\ifsagt\sagtfalse

\documentclass[runningheads]{llncs}
\usepackage{graphicx}
\usepackage{amsthm}

\usepackage{amsmath,amsfonts, mathdots, amssymb, latexsym, amsbsy,bm,mathtools}         % Mathematische Fonts amslatex
\usepackage{braket}
\usepackage[english]{babel}

\renewcommand\labelenumi{(\roman{enumi})}
\renewcommand\theenumi\labelenumi

\graphicspath{{figures/}}

\usepackage{cleveref}
\crefname{lemma}{Lemma}{Lemmas}
\crefname{theorem}{Theorem}{Theorems}

\usepackage{thmtools}
\usepackage{thm-restate}
\usepackage{soul}

\usepackage{textcomp} % for celsius

\usepackage{nicefrac}

\usepackage{dsfont}

\usepackage[colorinlistoftodos]{todonotes}

\newcommand{\R}{\mathbb{R}}

\newcommand{\N}{\mathbb{N}}

\newcommand*\diff{\mathop{}\!\mathrm{d}}

\DeclarePairedDelimiterX{\scalar}[2]{\langle}{\rangle}{#1, #2}

\renewcommand{\l}{\ell}

\newcommand{\Time}{[0, \infty)}
\newcommand{\Flow}{\R_{\geq 0}}

\newcommand{\capa}{\nu}

\renewcommand{\epsilon}{\varepsilon}

\newcommand{\source}{s}
\newcommand{\sink}{t}

\newcommand{\X}{X} % Mengen von statischen Flüssen oder Teilmenge in Lp Rauemen fuer VI

\newcommand{\inrate}{r} % network inflow rate

\newcommand{\T}{T} % Arc travel time

\graphicspath{{figures/}}

\begin{document}
\title{Dynamic Equilibria in Time-Varying Networks\thanks{Funded by the Deutsche Forschungsgemeinschaft (DFG, German Research Foundation) under Germany's Excellence Strategy – The Berlin Mathematics Research Center MATH+ (EXC-2046/1, project ID: 390685689).}}
\author{Hoang Minh Pham \and
Leon Sering%\orcidID{0000-0003-2953-1115}
}
\authorrunning{H. M. Pham and L. Sering}
\institute{Technische Universität Berlin, Germany\\
\email{hoang.m.pham@campus.tu-berlin.de} \qquad \email{sering@math.tu-berlin.de}
}
\maketitle              % typeset the header of the contribution
\begin{abstract}
Predicting selfish behavior in public environments by considering Nash equilibria is a central concept of game theory.
For the dynamic traffic assignment problem modeled by a flow over time game, in which every particle tries to reach its
destination as fast as possible, the dynamic equilibria are called Nash flows over time. So far, this model has only
been considered for networks in which each arc is equipped with a constant capacity, limiting the outflow rate, and with
a transit time, determining the time it takes for a particle to traverse the arc. However, real-world traffic networks
can be affected by temporal changes, for example, caused by construction works or special speed zones during some time
period. To model these traffic scenarios appropriately, we extend the flow over time model by time-dependent capacities
and time-dependent transit times. Our first main result is the characterization of the structure of Nash flows over
time. Similar to the static-network model, the strategies of the particles in dynamic equilibria can be characterized by
specific static flows, called thin flows with resetting. The second main result is the existence of Nash flows over
time, which we show in a constructive manner by extending a flow over time step by step by these thin flows.
\ifsagt 
\\[1em]
\textbf{Related Version:} A full version of this paper including all proofs is available at
\url{https://arxiv.org/abs/1807.05862}.
\fi
\end{abstract}
\section{Introduction} \label{sec:introduction} 
In the last decade the technological advances in the mobility and communication sector have grown rapidly enabling
access to real-time traffic data and autonomous driving vehicles in the foreseeable future. One of the major advantages
of self-driving and communicating vehicles is the ability to directly use information about the traffic network
including the route-choice of other road users. This holistic view of the network can be used to decrease travel times
and distribute the traffic volume more evenly over the network. As users will still expect to travel along a fastest
route it is important to incorporate game theoretical aspects when analyzing the dynamic traffic assignment. The
results can then be used by network designers to identify bottlenecks beforehand, forecast air pollution in dense urban
areas and give feedback on network structures.
In order to obtain a better understanding of the complicated interplay between traffic users it is important to develop
strong mathematical models which represent as many real-world traffic features as possible. Even though the more
realistic models consider a time-component, the network properties are considered to stay constant in most cases.
Surely, this is a serious drawback as real road networks often have properties that vary over time. For example, the
speed limit in school zones is often reduced during school hours, roads might be completely or partially blocked due to
construction work and the direction of reversible lanes can be switched, causing a change in the capacity in both
directions. A more exotic, but nonetheless important setting are evacuation scenarios. Consider an inhabited region of
low altitude with a high risk of flooding. As soon as there is a flood warning everyone needs to be evacuated to some
high-altitude-shelter. But, due to the nature of rising water levels, roads with low altitude will be impassable much
sooner than roads of higher altitude. In order to plan an optimal evacuation or simulate a chaotic equilibrium scenario
it is essential to use a model with time-varying properties. 
This research work is dedicated to providing a better understanding of the impact of dynamic road properties on the
traffic dynamics in the Nash flow over time model. We will transfer all essential properties of Nash flows over time in
static networks to networks with time-varying properties.

\subsection{Related Work} The fundamental concept for the model considered in this paper are \emph{flows over time} or
\emph{dynamic flows}, which were introduced back in 1956 by Ford and Fulkerson \cite{ford1958constructing,ford1962flows}
in the context of optimization problems. The key idea is to add a time-component to classical network flows, which means
that the flow particles need time to travel through the network. In 1959 Gale \cite{gale1959transient} showed the
existence of so called \emph{earliest arrival flows}, which solve several optimization problems at once, as they
maximize the amount of flow reaching the sink at all points in time simultaneously. Further work on these optimal flows
is due to Wilkinson~\cite{wilkinson1971algorithm}, Fleischer and Tardos \cite{fleischer1998efficient}, Minieka
\cite{minieka1973maximal} and many others. For formal definitions and a good overview of optimization problems in flow
over time settings we refer to the survey of Skutella \cite{skutella2009introduction}.

In order to use flows over time for traffic modeling it is important to consider game theoretic aspects. Some pioneer
work goes back to Vickrey \cite{vickrey1969congestion} and Yagar~\cite{yagar1971dynamic}. In the context of classical
(static) network flows, equilibria were introduced by Wardrop \cite{wardrop1952correspondence} in 1952. In 2009 Koch and
Skutella \cite{koch2009nash} (see also \cite{koch2011nash} and Koch's PhD thesis \cite{koch2012phd}) started a fruitful
research line by introducing dynamic equilibria, also called \emph{Nash flows over time}, which will be the central
concept in this paper. In a Nash flow over time every particle chooses a quickest path from the origin to the
destination, anticipating the route choice of all other flow particles. Cominetti et al.\ showed the existence of Nash
flows over time \cite{cominetti2011existence,cominetti2015dynamic} and studied the long term behavior
\cite{cominetti2017long}. Macko et al.~\cite{macko2013braess} studied the Braess paradox in this model and Bhaskar et
al.~\cite{bhaskar2015stackelberg} and Correa et al.~\cite{correa2019price} bounded the price of anarchy under certain
conditions. In 2018 Sering and Skutella \cite{sering2018multiterminal} transferred Nash flows over time to a model with
multiple sources and multiple sinks and in the following year Sering and Vargas Koch \cite{sering2019nash} considered
Nash flows over time in a model with spillback.

A different equilibrium concept in the same model was considered by Graf et al.~\cite{graf2020ide} by introducing
instantaneous dynamic equilibria. In these flows over time the particles do not anticipate the further evolution of the
flow, but instead reevaluate their route choice at every node and continue their travel on a current quickest path. In
addition to that, there is an active research line on packet routing games. Here, the traffic agents are modeled by
atomic packets (vehicles) of a specific size. This is often combined with discrete time steps. Some of the recent work
on this topic is due to Cao et al.~\cite{cao2017network}, Harks et al.~\cite{harks2018competitive}, Peis et
al.~\cite{peis2018oligopolistic} and Scarsini et al.~\cite{scarsini2018dynamic}.

\subsection{Overview and Contribution} In the \emph{base model}, which was considered by Koch and Skutella
\cite{koch2011nash} and by the follow up research
\cite{bhaskar2015stackelberg,cominetti2011existence,cominetti2015dynamic,cominetti2017long,correa2019price,macko2013braess,sering2018multiterminal},
the network is constant and each arc has a constant capacity and constant transit time. In real-world traffic, however,
temporary changes of the infrastructure are omnipresent. In order to represent this, we extend the base model to
networks with time-varying capacities (including the network inflow rate) and time-varying transit times.

We start in \Cref{sec:flow_dynamics} by defining the flow dynamics of the deterministic queuing model with time-varying
arc properties and proving some first auxiliary results. In particular, we describe how to turn time-dependent speed
limits into time-dependent transit times. In \Cref{sec:nash_flows} we introduce some essential properties, such as the
earliest arrival times, which enable us to define Nash flows over time. As in the base model, it is still possible to
characterize such a dynamic equilibrium by the underlying static flow. Taking the derivatives of these parametrized
static flows provides thin flows with resetting, which are defined in \Cref{sec:thin_flows}. We show that the central
results of the base model transfer to time-varying networks, and in particular, that the derivatives of every Nash flow
over time form a thin flow with resetting. In \Cref{sec:construction} we show the reverse of this statement: Nash flows
over time can be constructed by a sequence of thin flows with resetting, which, in the end, proves the existence of
dynamic equilibria. We close this section with a detailed example. Finally, in \Cref{sec:conclusion} we present a
conclusion and give a brief outlook on further research directions.

\ifsagt
Due to space restrictions we omit most technical proofs. They can be found in the appendix of the full version available at
\url{https://arxiv.org/abs/1807.05862}.
\fi

\section{Flow Dynamics} \label{sec:flow_dynamics}

We consider a directed graph $G = (V, E)$ with a source~$\source$ and a sink~$\sink$ such that each node is reachable by
$\source$. In contrast to the Koch-Skutella model, which we will call \emph{base model} from now on, this time each arc
$e$ is equipped with a time-dependent capacity $\capa_e \colon \Time \to (0, \infty)$ and a time-dependent speed
limit~$\lambda_e \colon \Time \to (0, \infty)$, which is inversely proportional to the transit time. We consider a
time-dependent network inflow rate $\inrate\colon \Time \to \Time$ denoting the flow rate at which particles enter the
network at $\source$. We assume that the amount of flow an arc can support is unbounded and that the network inflow is
unbounded as well, i.e., for all $e \in E$ we require that
\[\int_0^\theta \capa_e(\xi) \diff \xi \to \infty, \quad \int_0^\theta \lambda_e(\xi) \diff \xi \to \infty \quad 
\text{ and } \quad \int_0^\theta \inrate(\xi) \diff \xi \to \infty \quad \text{ for } \theta \to \infty.\]
Later on, in order to be able to construct Nash flows over time, we will additionally assume that all these functions
are right-constant, i.e., for every $\theta \in \Time$ there exists an $\epsilon > 0$ such that the function is constant
on $[\theta, \theta + \epsilon)$.

\subsubsection{Speed limits.} Let us focus on the transit times first. We have to be careful how to model the
transit time changes, since we do not want to lose the following two properties of the base model:
\begin{enumerate}
\item We want to have the first-in-first-out (FIFO) property for arcs, which leads to FIFO property of the network for
Nash flows over time \cite[Theorem 1]{koch2011nash}. 
\item Particles should never have the incentive to wait on a node.
\end{enumerate}
In other words, we cannot simply allow piecewise-constant transit times, since this could lead to the following case: If
the transit time of an arc is high at the beginning and gets reduced to a lower value at some later point in time,
then particles might overtake other particles on that arc. Thus, particles might arrive earlier at the sink if they wait
right in front of the arc until its transit time drops.
Hence, we let the speed limit change over time instead. In order to keep the number of parameters of the network as
small as possible, we assume that the lengths of all arcs equal $1$ and, instead of a transit time, we equip every arc~$e
\in E$ with a time-dependent speed limit $\lambda_e \colon \Time \to (0, \infty)$. Thus, a particle might traverse the
first part of an arc at a different speed than the remaining distance if the maximal speed changes midway; see
\Cref{fig:speed_limits_on_road:timedep}.
\begin{figure}[t]
\centering \includegraphics[page=1]{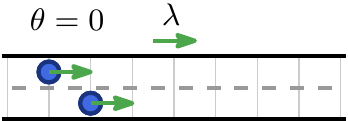} \quad
\includegraphics[page=2]{speed_limits_on_road_timedep} \quad \includegraphics[page=3]{speed_limits_on_road_timedep}
\caption{Consider a road segment with time-dependent speed limit that is low in the time interval $[0, 1)$ and large
afterwards. All vehicles, independent of their position, first traverse the link slowly and immediately speed up to the
new speed limit at time~$1$.} \label{fig:speed_limits_on_road:timedep}
\end{figure}
\subsubsection{Transit times.} Note that we assume the point queue of an arc to always right in front of
the exit. Hence, a particle entering arc $e$ at time $\theta$ immediately traverses the arc of length $1$ with a
time-dependent speed of $\lambda_e$. The \emph{transit time}~${\tau \colon \Time \to \Time} $ is therefore given by
\[\tau_e(\theta) \coloneqq \min \Set{ \tau \geq 0 | \int_{\theta}^{\theta+\tau} \lambda_e(\xi) \diff \xi = 1}.\]
Since we required $\int_0^\theta \lambda_e(\xi) \diff \xi$ to be unbounded for $\theta \to \infty$, we always have a
finite transit time. For an illustrative example see \Cref{fig:from_speed_limits_to_transit_times:timedep}. 
\pagebreak[2]

\begin{figure}[t]
\begin{minipage}{0.48\textwidth}
\centering \includegraphics{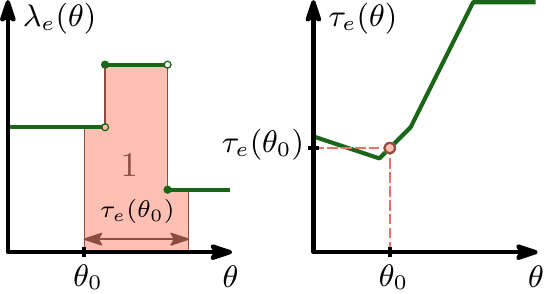} \vspace{-0.5\baselineskip}
\caption{From speed limits (\emph{left side}) to transit times (\emph{right side}). The transit time~$\tau_e(\theta)$
denotes the time a particle needs to traverse the arc when entering at time~$\theta$. We normalize the speed limits by
assuming that all arcs have length~$1$, and hence, the transit time $\tau_e(\theta)$ equals the length of an interval
starting at $\theta$ such that the area under the speed limit graph within this interval is $1$.}
\label{fig:from_speed_limits_to_transit_times:timedep}
\end{minipage} \hfill
\begin{minipage}{0.48\textwidth}
\centering \vspace{\baselineskip} \includegraphics{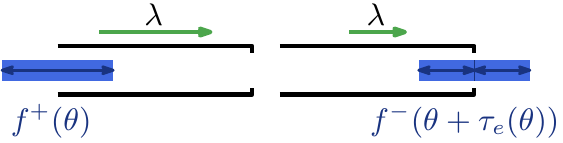}\vspace{.5\baselineskip} \caption{An illustration
of how the flow rate changes depending on the speed limits. \emph{On the left:} As the speed limit $\lambda$ is high,
the flow volume entering the arc per time unit is represented by the area of the long rectangle. \emph{On the right:}
The speed limit is halved, and therefore, the same amount of flow needs twice as much time to leave the arc (or enter
the queue if there is one). Hence, if there is no queue, the outflow rate at time $\tau + \tau_e(\theta)$ is only half
the size of the inflow rate at time $\theta$.} \label{fig:speed_ratio:timedep}
\end{minipage}
\end{figure}

The following lemma shows some basic properties of the transit times.
\begin{restatable}{lemma}{transittimestimedep} \label{lem:transit_times:timedep}
For all $e \in E$ and almost all $\theta \in \Time$ we have:
\begin{enumerate}
\item The function $\theta \mapsto \theta + \tau_e(\theta)$ is strictly increasing. \label{it:queueing_time_monotone:timedep}
\item The function $\tau_e$ is continuous and almost everywhere differentiable. \label{it:tau_is_diffbar:timedep}
\item For almost all $\theta \in \Time$ we have
$1 + \tau'_e(\theta) = \frac{\lambda_e(\theta)}{\lambda_e(\theta + \tau_e(\theta))}$. \label{it:derivative_of_tau:timedep}
\end{enumerate}
\end{restatable}
These statement follow by simple computation and some basic Lebesgue integral theorems. 
\ifarxiv 
The proof can be found in the appendix on page \pageref{proof:transit_times:timedep}.
\fi

\subsubsection{Speed ratios.} The ratio in \Cref{lem:transit_times:timedep}~\ref{it:derivative_of_tau:timedep} will be
important to measure the outflow of an arc depending on the inflow. We call $\gamma_e \colon \Time \to [0, \infty)$ the
\emph{speed ratio} of $e$ and it is defined by
$\gamma_e(\theta) \coloneqq \frac{\lambda_e(\theta)}{\lambda_e(\theta + \tau_e(\theta))} = 1 + \tau'_e(\theta)$.
Figuratively speaking, this ratio describes how much the flow rate changes under different speed limits. If, for
example, $\gamma_e(\theta) = 2$, as depicted in \Cref{fig:speed_ratio:timedep}, this means that the speed limit was
twice as high when the particle entered the arc as it is at the moment the particle enters the queue. In this case the
flow rate is halved on its way, since the same amount of flow that entered within one time unit, needs two time units to
leave it. With the same intuition the flow rate is increased whenever $\gamma_e(\theta) < 1$. Note that in figures of
other publications on flows over time the flow rate is often pictured by the width of the flow. But for time-varying
networks this is not accurate anymore as the transit speed can vary. Hence, in this paper the flow rates are given by
the width of the flow multiplied by the current speed limit.

%\paragraph{Flows over time.}
A \emph{flow over time} is specified by a family of locally integrable and bounded functions $f = (f^+_e, f^-_e)_{e\in
E}$ denoting the in- and outflow rates. The \emph{cumulative in- and outflows} are given by
\[F_e^+(\theta) \coloneqq \int_0^\theta f_e^+(\xi) \diff \xi \qquad \text{ and } \qquad
F_e^-(\theta) \coloneqq \int_0^\theta f_e^-(\xi) \diff \xi.\]
A flow over time \emph{conserves flow on all arcs $e$}:
\begin{equation}\label{eq:conservation_on_arcs:timedep}
F_e^-(\theta + \tau_e(\theta)) \leq F_e^+(\theta) \qquad \text{for all } \theta \in [0, \infty],
\end{equation}
and \emph{conserves flow at every node~$v \in V\backslash \set{\sink}$} for almost all $\theta \in \Time$:
\begin{equation} \label{eq:flow_conservation:timedep}
\sum_{e\in \delta^+_v} f_e^+(\theta) - \sum_{e \in \delta^-_v} f_e^-(\theta) = \begin{cases}
0 & \text{ if } v \in V\setminus \set{\sink}, \\
\inrate(\theta) & \text{ if } v = \source.
\end{cases}\end{equation}

%\paragraph{Queues.} 
A particle entering an arc $e$ at time $\theta$ reaches the head of the arc at time $\theta +
\tau_e(\theta)$ where it lines up at the point queue. Thereby, the \emph{queue size} $z_e\colon \Time \to \Time$ at
time $\theta + \tau_e(\theta)$ is defined by
$z_e(\theta + \tau_e(\theta)) \coloneqq F_e^+(\theta) - F_e^-(\theta + \tau_e(\theta))$.

%\paragraph{Feasibility.} 
We call a flow over time in a time-varying network \emph{feasible} if we have for almost
all $\theta \in \Time$ that
\begin{equation}\label{eq:outflow:timedep}
f_e^-(\theta + \tau_e(\theta)) = \begin{cases}
\capa_e(\theta + \tau_e(\theta)) & \text{ if } z_e(\theta + \tau_e(\theta)) > 0, \\
\min\Set{\frac{f_e^+(\theta)}{\gamma_e(\theta)}, \capa_e(\theta + \tau_e(\theta))} & \text{ else, }
\end{cases} \end{equation}
and $f_e^-(\theta) = 0$ for almost all $\theta < \tau_e(0)$. 

Note that the outflow rate depends on the speed ratio $\gamma_e(\theta)$ if the queue is empty (see
\Cref{fig:speed_ratio:timedep}). Otherwise, the particles enter the queue, and therefore, the outflow rate equals the
capacity independent of the speed ratio. Furthermore, we observe that every arc with a positive queue always has a
positive outflow, since the capacities are required to be strictly positive. And finally, \eqref{eq:outflow:timedep}
implies \eqref{eq:conservation_on_arcs:timedep}, which can easily be seen by
considering the derivatives of the cumulative flows whenever we have an empty queue, i.e., $F_e^-(\theta +
\tau_e(\theta)) = F_e^+(\theta)$. By \eqref{eq:outflow:timedep} we have that ${f_e^-(\theta + \tau_e(\theta)) \cdot (1 +
\tau'_e(\theta)) \leq f_e^+(\theta)}$. Hence, \eqref{eq:flow_conservation:timedep} and
\eqref{eq:outflow:timedep} are sufficient for a family of functions $f = (f_e^+, f_e^-)_{e \in E}$ to be a feasible flow
over time.

 \begin{figure}[t] 
 \centering \includegraphics{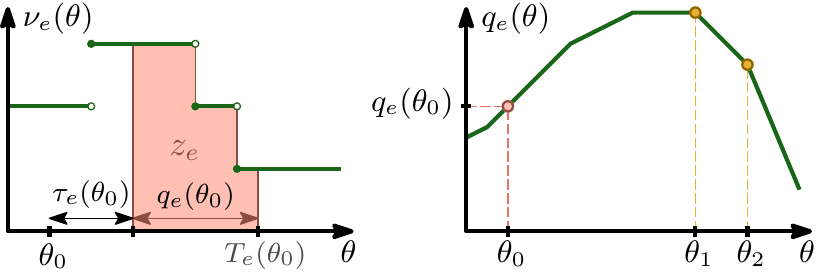} \caption{Waiting times for time-dependent capacities. The waiting
 time of a particle $\theta_0$ (\emph{right side}) is given by the length of the interval starting at $\theta_0
 +\tau_e(\theta_0)$ such that the area underneath the capacity graph equals the queue size at time~$\theta_0
 +\tau_e(\theta_0)$ (\emph{left~side}). The right boundary of the interval equals the exit time $T_e(\theta_0)$. The
 waiting time does not only depend on the capacity but also on the inflow rate and the transit times. For example, if
 the capacity and the speed limit are constant but the inflow rate is $0$, the waiting time will decrease with a slope
 of $1$ (\emph{right side} within $[\theta_1, \theta_2]$).} \label{fig:waiting_times:timedep}
 \end{figure}

%\paragraph{Waiting times.} 
The \emph{waiting time} $q_e \colon \Time \to \Time$ of a particle that enters the arc at
time $\theta$ is defined by
\[q_e(\theta) \coloneqq \min \Set{q \geq 0 | \int_{\theta + \tau_e(\theta)}^{\theta + \tau_e(\theta) + q} 
\capa_e(\xi) \diff \xi = z_e(\theta + \tau_e(\theta))}.\]
As we required $\int_0^{\theta} \capa_e(\xi) \diff \xi$ to be unbounded for $\theta \to \infty$ the set on the right
side is never empty. Hence, $q_e(\theta)$ is well-defined and has a finite value. In addition, $q_e$ is continuous since
$\capa_e$ is always strictly positive. 
%
%\paragraph{Exit times.} 
The \emph{exit time} $T_e \colon \Time \to \Time$ denotes the time at which the particles
that have entered the arc at time $\theta$ finally leave the queue. Hence, we define
$T_e(\theta) \coloneqq \theta + \tau_e(\theta) + q_e(\theta)$.
In \Cref{fig:waiting_times:timedep} we display an illustrative example for the
definition of waiting and exit times.

With these definitions we can show the following lemma.  
\begin{restatable}{lemma}{technicalpropertiestimedep}\label{lem:technical_properties:timedep}
For a feasible flow over time $f$ it holds for all $e \in E$, $v \in V$ and ${\theta \in \Time}$ that:
\begin{enumerate}
\item $q_e(\theta) > 0 \;\; \Leftrightarrow \;\; z_e(\theta+\tau_e(\theta)) > 0$. \label{it:q_equiv_z:timedep}
\item $z_e(\theta +\tau_e(\theta) + \xi) > 0$ for all $\xi \in [0, q_e(\theta))$.
\label{it:positive_queue_while_emptying:timedep}
\item $F^+_e(\theta)=F^-_e(T_e(\theta))$. \label{it:in_equals_out_at_exit_time:timedep}
\item For $\theta_1 < \theta_2$ with $F_e^+(\theta_2) - F_e^+(\theta_1) = 0$ and $z_e(\theta_2 + \tau_e(\theta_2))>0$  
we have $T_e(\theta_1)=T_e(\theta_2)$. \label{it:equal_exit_times:timedep}
\item The functions $T_e$ are monotonically increasing. \label{it:T_monoton:timedep}
\item The functions $q_e$ and $T_e$ are continuous and almost everywhere differentiable.\label{it:q_is_diffbar:timedep}
\item For almost all $\theta \in \Time$ we have
\[T'_e(\theta) = \begin{cases}
\frac{f_e^+(\theta)}{\capa_e(T_e(\theta))}& \text{ if } q_e(\theta) > 0,\\
\max \Set{\gamma_e(\theta), \frac{f_e^+(\theta)}{\capa_e(T_e(\theta))}} & \text{ else.}
\end{cases}\] \label{it:derivative_of_T:timedep}
\end{enumerate}  
\end{restatable}
Most of the statements follow immediately from the definitions and some involve minor calculations. For
\ref{it:q_is_diffbar:timedep} we use Lebesgue's differentiation theorem. 
\ifarxiv
As the proof does not give any interesting further insights we moved it to the appendix on page
\pageref{proof:technical_properties:timedep}.
\fi

\section{Nash Flows Over Time}\label{sec:nash_flows} In order to define a dynamic equilibrium we consider the particles
as players in a dynamic game. For this the set of particles is identified by the non-negative reals denoted by $\Flow$.
The flow volume is hereby given by the Lebesgue-measure, which means that $[a, b] \subseteq \Flow$ with $a < b$ contains
a flow volume of $b - a$. The flow particles enter the network according to the ordering of the reals beginning with
particle $0$. It is worth noting that a particle $\phi \in \Flow$ can be split up further so that for example one half
takes a different route than the other half. As characterized by Koch and Skutella, a dynamic equilibrium is a feasible
flow over time, where almost all particles only use current shortest paths from~$\source$ to~$\sink$. Note that we
assume a game with full information. Consequently, all particles know all speed limit and capacity functions in advance
and have the ability to perfectly predict the future evolution of the flow over time. Hence, each particle perfectly
knows all travel times and can choose its route accordingly. We start by defining the earliest arrival times for a
particle $\phi \in \Flow$.

%\paragraph{Earliest arrival times.}
The \emph{earliest arrival time functions} $\l_v \colon \Flow \to \Time$ map each particle $\phi$ to the earliest time
$\l_v(\phi)$ it can possibly reach node $v$. Hence, it is the solution to
\begin{equation} \label{eq:l_v_def:timedep}
\l_v(\phi) = \begin{cases}
\quad \!\min \;\;\Set{ \theta \geq 0 | \int_0^\theta r(\xi) \diff \xi = \phi} & \text{ for } v = \source,\\
\min\limits_{e = uv \in \delta_v^-} T_e(\l_u(\phi)) & \text{ else.}
\end{cases}
\end{equation}
%As we have finite speed limits we have that $T_e(\l_u(\theta))$ is strictly larger than~$\l_u(\theta)$.

Note that for all $v \in V$ the earliest arrival time function $\l_v$ is non-decreasing, continuous and almost
everywhere differentiable. This holds directly for $\l_{\source}$ and for $v \neq \source$ it follows inductively, since
these properties are preserved by the concatenation $T_e \circ \l_u$ and by the minimum of finitely many functions.

%\paragraph{Active arcs, current shortest paths networks and resetting arcs.} 
For a particle~$\phi$ we call an arc $e =
uv$ \emph{active} if $\l_v(\phi) = T_e(\l_u(\phi))$. The set of all these arcs are denoted by $E'_\phi$ and these are
exactly the arcs that form the current shortest paths from $s$ to some node $v$. For this reason we call the subgraph
$G'_\phi = (V, E'_\phi)$ the \emph{current shortest paths network} for particle $\phi$. Note that $G'_\phi$ is acyclic
and that every node is reachable by $\source$ within this graph. The arcs where particle $\phi$ experiences a waiting
time when traveling along shortest paths only are called \emph{resetting arcs} denoted by $E^*_\phi \coloneqq \Set{e =
uv \in E | q_e(\l_u(\phi)) > 0}$.

%\paragraph{Dynamic equilibria.} 
Nash flows over time in time-varying networks are defined in the exact same way as
Cominetti et al.\ defined them in the base model \cite[Definition 1]{cominetti2011existence}.
\begin{definition}[Nash flow over time]\label{def:nash_flows:timedep}
We call a feasible flow over time $f$ a \emph{Nash flow over time} if the following \emph{Nash flow condition}
holds:
\begin{equation}
f^+_e(\theta) > 0 \;\;\Rightarrow\;\; \theta \in \l_u(\Phi_e) \;\;\text{ for all } e = uv \in E \text{ and almost
all } \theta \in \Time,\tag{N}\label{eq:Nash_condition:timedep}
\end{equation}
 where $\Phi_e \coloneqq \set{\phi \in \Flow | e \in E'_\phi}$ is the set of particles for which arc $e$ is active. 
\end{definition}

As Cominetti et al.\ showed in \cite[Theorem 1]{cominetti2015dynamic} these Nash flows over time can be characterized as
follows.

\begin{lemma} \label{lem:nash_flow_characterization:timedep}
A feasible flow over time $f$ is a Nash flow over time if, and only if, for all $e = uv \in E$ and all $\phi \in \Flow$
we have ${F_e^+(\l_u(\phi)) = F_e^-(\l_v(\phi))}$.
\end{lemma}

Since the exit and the earliest arrival times have the same properties in time-varying networks as in the base model,
this lemma follows with the exact same proof that was given by Cominetti et al.\ for the base model \cite[Theorem
1]{cominetti2015dynamic}. The same is true for the following lemma; see \cite[Proposition 2]{cominetti2015dynamic}.

{\samepage
\begin{lemma}\label{lem:resetting_implies_active:timedep}
Given a Nash flow over time the following holds for all particles $\phi$:
\begin{enumerate}
  \item $E^*_\phi \subseteq E'_\phi$. \label{it:resetting_subset_active:timedep}
  \item $E'_\phi = \set{e = uv | \l_v(\phi) \geq \l_u(\phi) + \tau_e(\theta)}$. \label{it:characterization_of_active:timedep}
  \item $E^*_\phi = \set{e = uv | \l_v(\phi) > \l_u(\phi) + \tau_e(\theta)}$.
  \label{it:characterization_of_resetting:timedep}
\end{enumerate}
\end{lemma} }
 %\paragraph{Underlying static flows.} 
 Motivated by \Cref{lem:nash_flow_characterization:timedep} we define the
 \emph{underlying static flow} for $\phi \in \Flow$ by
 \[x_e(\phi) \coloneqq F_e^+(\l_u(\phi)) = F_e^-(\l_v(\phi)) \quad \text{ for all } e = uv \in E.\]
 By the definition of $\l_\source$ and the integration of \eqref{eq:flow_conservation:timedep} we have
 $\int_0^{\l_{\source}(\phi)}\inrate(\xi) \diff \xi = \phi$, and hence, $x_e(\phi)$ is a static $\source$-$\sink$-flow
 (classical network flow) of value $\phi$, whereas the derivatives $(x'_e(\phi))_{e \in E}$ form a static
 $\source$-$\sink$-flow of value $1$.

\section{Thin Flows} \label{sec:thin_flows} 
\emph{Thin flows with resetting}, introduced by Koch and
Skutella~\cite{koch2011nash}, characterize the derivatives $(x'_e)_{e \in E}$ and $(\l'_v)_{v \in V}$ of Nash flows over
time in the base model. In the following we will transfer this concept to time-varying networks.

Consider an acyclic network $G' = (V, E')$ with a source $\source$ and a sink $\sink$, such that every node is reachable
by~$\source$. Each arc is equipped with a capacity $\capa_e > 0$ and a speed ratio $\gamma_e > 0$. Furthermore, we have
a network inflow rate of $\inrate > 0$ and an arc set $E^*\subseteq E'$. We obtain the following definition.

\begin{definition}[Thin flow with resetting in a time-varying network]\label{def:thin_flows:timedep} 
A~static $\source$-$\sink$ flow $(x'_e)_{e \in E}$ of value $1$ together with a node labeling $(\l'_v)_{v \in V}$ is a
\emph{thin flow with resetting} on $E^*$ if:
\begin{alignat}{2}
\l'_{\source} &= \frac{1}{r} && \label{eq:l'_s:timedep} \tag{TF1}\\ 
\l'_v &= \min_{e = uv \in E'} \rho_e(\l'_u, x'_e) \quad && \text{ for all } v \in V \setminus \set{\source},
\label{eq:l'_v_min:timedep}\tag{TF2}\\ 
\l'_v &= \rho_e(\l'_u, x'_e) && \text{ for all } e = uv \in E' \text{ with } x'_e > 0, \label{eq:l'_v_tight:timedep}\tag{TF3}
\end{alignat}
\[\text{ where } \qquad \rho_e(\l'_u, x'_e) \coloneqq \begin{cases}
\frac{x'_e}{\capa_e} & \text{ if } e = uv \in E^*,\\
\max\Set{\gamma_e \cdot \l'_u, \frac{x'_e}{\capa_e}} & \text{ if } e = uv \in E'\backslash E^*.  
\end{cases}\]
\end{definition}
  
The derivatives of a Nash flow over time in time-varying networks do indeed form a thin flow with resetting as the
following theorem shows.

\begin{theorem} \label{thm:derivatives_form_thin_flow:timedep}
   For almost all~$\phi \!\in\! \Flow$ the derivatives $(x'_e(\phi))_{e\in E'_\phi}$ and $(\l'_v(\phi))_{v\in V}$ of a
   Nash flow over time $f = (f^+_e, f_e^-)_{e\in E}$ form a thin flow with resetting on~$E^*_\phi$ in the current
   shortest paths network $G'_\phi = (V, E'_\phi)$ with network inflow rate~$\inrate(\l_\source(\phi))$ as well as
   capacities~$\capa_e(\l_v(\phi))$ and speed ratios~$\gamma_e(\l_u(\phi))$ for each arc~$e = uv \in E$.
\end{theorem}
\begin{proof}
Let $\phi \in \Flow$ be a particle such that for all arcs $e = uv \in E$ the derivatives of $x_e$, $\l_u$, $\T_e \circ
\l_u$ and $\tau_e$ exist and $x_e'(\phi)  = f_e^+(\l_u(\phi)) \cdot \l'_u(\phi) = f_e^-(\l_v(\phi)) \cdot \l'_v(\phi)$ as
well as $1 + \tau'_e(\l_u(\phi)) = \gamma_e(\l_u(\phi))$. This is given for almost all $\phi$.

By \eqref{eq:l_v_def:timedep} we have $\int_0^{\l_{\source}(\phi)} \inrate(\xi) \diff \xi = \phi$ and taking the
derivative by applying the chain rule, yields $\inrate(\l_{\source}(\phi)) \cdot \l'_\source(\phi) = 1$, which shows
\eqref{eq:l'_s:timedep}.

Taking the derivative of \eqref{eq:l_v_def:timedep} at time $\l_u(\phi)$ by using the differentiation rule for a minimum
\ifarxiv 
(\Cref{lem:diff_rule_for_min:pre} in the appendix) 
\fi
yields
$\l'_v(\phi) = \min_{e = uv \in E'} T'_e(\l_u(\phi)) \cdot \l'_u(\phi)$. 
By using \Cref{lem:technical_properties:timedep} \ref{it:derivative_of_T:timedep} we obtain
\begin{align*}T'_e(\l_u(\phi)) \cdot \l'_u(\phi) &= \begin{cases} \frac{f_e^+(\l_u(\phi))}{\capa_e(T_e(\l_u(\phi)))} 
\cdot \l'_u(\phi) & \text{if } q_e(\l_u(\phi)) > 0,\\
\max\Set{\gamma_e(\l_u(\phi)) , \frac{f_e^+(\l_u(\phi))}{\capa_e(T_e(\l_u(\phi)))}} \cdot \l'_u(\phi)& \text{else,}
\end{cases}\\
&= \rho_e(\l'_u(\phi), x'_e(\phi)),\end{align*}
which shows \eqref{eq:l'_v_min:timedep}.
  
Finally, in the case of $f_e^-(\l_v(\phi)) \cdot \l'_v(\phi) = x'_e(\phi) > 0$ we have by \eqref{eq:outflow:timedep}
that
\begin{align*}
\l'_v(\phi) = \frac{x'_e(\phi)}{f_e^-(\l_v(\phi))} &= \begin{dcases}
\frac{x'_e(\phi)}{\min\Set{\frac{f^+_e(\l_u(\phi))}{\gamma_e(\l_u(\phi))},\capa_e(\l_v(\phi))}}\, 
&\text{ if } q_e(\l_u(\phi)) = 0,\\
\frac{x'_e(\phi)}{\capa_e(\l_v(\phi))} &\text{ else},
\end{dcases}\\
&=\begin{cases}
\max\Set{\gamma_e(\l_u(\phi)) \cdot \l'_u(\phi), \frac{x'_e(\phi)}{\capa_e(\l_v(\phi))}} 
&\text{ if } e \in E'_\phi \backslash E^*_\phi,\!\!\\
\frac{x'_e(\phi)}{\capa_e(\l_v(\phi))} &\text{ if } e \in E^*_\phi,
\end{cases} \\
&= \rho_e(\l'_u(\phi), x'_e(\phi)).
\end{align*}
This shows \eqref{eq:l'_v_tight:timedep} and finishes the proof. 
\end{proof}

In order to construct Nash flows over time in time-varying networks, we first have to show that there always exists a
thin flow with resetting.

\begin{restatable}{theorem}{existenceofthinflowstimedep} \label{thm:existence_of_thin_flows:timedep}
  Consider an acyclic graph~$G' = (V, E')$ with source~$\source$, sink~$\sink$, capacities~$\capa_e > 0$, speed
  ratios~$\gamma_e > 0$ and a subset of arcs~$E^* \subseteq E'$, as well as a network inflow $\inrate > 0$. Furthermore,
  suppose that every node is reachable from~$\source$. Then there exists a thin flow $\left((x'_e)_{e \in E},
  (\l'_v)_{v\in V}\right)$ with resetting on~$E^*$.
\end{restatable}

This proof works exactly as the proof for the existence of thin flows in the base model presented by Cominetti et al.\
\cite[Theorem 3]{cominetti2015dynamic}. 
\ifarxiv
In addition, a detailed proof utilizing Kakutani's fixed point theorem is given in the appendix.
\fi

\section{Constructing Nash Flows Over Time} \label{sec:construction}
In the remaining part of this paper we assume that for all $e \in E$ the functions~$\capa_e$ and $\lambda_e$ as well as
the network inflow rate function $\inrate$ are right-constant. In order to show the existence of Nash flows over time in
time-varying networks we use the same $\alpha$-extension approach as used by Koch and Skutella in \cite{koch2011nash}
for the base model. The key idea is to start with the empty flow over time and expand it step by step by using a thin
flow with resetting.

%\paragraph{$\alpha$-Extension.} 
Given a \emph{restricted Nash flow over time} $f$ on $[0, \phi]$, i.e., a Nash flow
over time where only the particles in $[0, \phi]$ are considered, we obtain well-defined earliest arrival times
$(\l_v(\phi))_{v\in V}$ for particle $\phi$. Hence, by \Cref{lem:resetting_implies_active:timedep} we can determine the
current shortest paths network $G'_\phi = (V, E'_\phi)$ with the resetting arcs~$E^*_\phi$, the capacities
$\capa_e(\l_v(\phi))$ and speed ratios $\gamma_e(\l_u(\phi))$ for all arcs~${e=uv\in E'}$ as well as the network inflow rate
$\inrate(\l_\source(\phi))$. By \Cref{thm:existence_of_thin_flows:timedep} there exists a thin flow $\left((x'_e)_{e \in
E'}, (\l'_v)_{v \in V}\right)$ on $G'_\phi$ with resetting on $E^*_\phi$. For $e \not\in E'_\phi$ we set $x'_e \coloneqq
0$. We extend the $\l$- and $x$-functions for some $\alpha> 0$ by
\[\l_v(\phi + \xi) \coloneqq \l_v(\phi) + \xi \cdot \l'_v \quad \text{and} \quad
x_e(\phi) \coloneqq x_e(\phi) + \xi \cdot x'_e \qquad \text{for all } \xi \in [0, \alpha)\]
and the in- and outflow rate functions by
\begin{align*}
f_e^+(\theta) \coloneqq \frac{x'_e}{\l'_u} \text{ for } \theta \in [\l_u(\phi), \l_u(\phi \!+\! \alpha));\qquad\!\!
  f_e^-(\theta) \coloneqq \frac{x'_e}{\l'_v} \text{ for } \theta \in [\l_v(\phi), \l_v(\phi \!+\! \alpha)).\end{align*}
We call this extended flow over time \emph{$\alpha$-extension}. Note that $\l'_u = 0$ means that $[\l_u(\phi),
\l_u(\phi + \alpha))$ is empty, and the same holds for $\l'_v$.

%\paragraph{Feasible extension step size.} 
An $\alpha$-extension is a restricted Nash flow over time, which we will prove
later on, as long as the $\alpha$ stays within reasonable bounds. Similar to the base model we have to ensure that
resetting arcs stay resetting and non-active arcs stay non-active for all particles in $[\phi, \phi + \alpha)$. Since
the transit times may now vary over time, we have the following conditions for all $\xi \in [0, \alpha)$:

\begin{align}
\l_v(\phi) + \xi \cdot \l'_v - \l_u(\phi) - \xi \cdot \l'_u &> \tau_e(\l_u(\phi) + \xi \cdot \l'_u)) \quad 
\text{ for every } e \in E^*_\phi,\label{eq:alpha_resetting:timedep}\\
\l_v(\phi) + \xi \cdot \l'_v - \l_u(\phi) - \xi \cdot \l'_u &< \tau_e(\l_u(\phi) + \xi \cdot \l'_u)) \quad 
\text{ for every } e \in E \setminus E'_\phi.\label{eq:alpha_others:timedep}
\end{align}

Furthermore, we need to ensure that the capacities of all active arcs and the network inflow rate do not change within
the phase:
\begin{align}
\capa_e(\l_v(\phi)) &= \capa_e(\l_v(\phi) + \xi \cdot \l'_v) \quad \text{ for every } e \in E'_\phi \text{ and all } 
\xi \in [0, \alpha). \label{eq:alpha_capacity:timedep}\\
\inrate(\l_\source(\phi)) &= \inrate(\l_\source(\phi) + \xi \cdot \l'_\source) \quad \;\; \text{ for all } 
\xi \in [0, \alpha). \label{eq:alpha_inrate:timedep}
\intertext{Finally, the speed ratios need to stay constant for all active arcs, i.e.,}
\gamma_e(\l_u(\phi)) &= \gamma_e(\l_u(\phi) + \xi \cdot \l'_u) \quad \text{ for every } e \in E'_\phi \text{ and all } 
\xi \in [0, \alpha). \label{eq:alpha_speed_labels:timedep}
\end{align}

We call an $\alpha > 0$ \emph{feasible} if it satisfies \eqref{eq:alpha_resetting:timedep} to
\eqref{eq:alpha_speed_labels:timedep}.

%\begin{remark}
%The Conditions \eqref{eq:alpha_resetting:timedep} to \eqref{eq:alpha_speed_labels:timedep} on $\alpha$ are sufficient to
%guarantee that the $\alpha$-~extension is a Nash flow over time as we will show in \Cref{thm:extension:timedep}. But it
%is worth noting that it is possible to formulate more complex conditions on $\alpha$ in order to have longer thin flow
%phases by skipping events that do not change the thin flow. For example, \Cref{eq:alpha_speed_labels:timedep} only needs
%to hold for arcs that stay active during the phase, i.e., arcs for which \eqref{eq:l'_v_tight:timedep} is satisfied. But
%if we restrict it to those, we need to ensure that active arcs leaving the current shortest path network immediately,
%also satisfy \Cref{eq:alpha_others:timedep} for all $\xi \in (0, \alpha)$ as these arcs might become active again within
%the same phase  due to a change of the speed ratio. As these improved conditions are more complicated to state we will
%only consider the conditions stated above in this paper. But for an actual implementation there is room for improvement
%as a reduction in the number of the costly thin flow computations can speed up the algorithm significantly.
%\end{remark}

\begin{lemma}
Given a restricted Nash flow over time $f$ on $[0, \phi]$ then for right-constant capacities and speed limits there
always exists a feasible $\alpha > 0$.
\end{lemma}
\begin{proof}
By \Cref{lem:resetting_implies_active:timedep} we have that $\l_v(\phi) - \l_u(\phi) > \tau_e(\phi)$ for all $e \in
E^*_\phi$ and $\l_v(\phi) - \l_u(\phi) < \tau_e(\phi)$ for all $e \in E \setminus E'_\phi$. Since $\tau_e$ is continuous
there is an $\alpha_1 > 0$ such that \eqref{eq:alpha_resetting:timedep} and \eqref{eq:alpha_others:timedep} are
satisfied for all $\xi \in [0, \alpha_1)$. Since $\capa_e$, $\inrate$ and $\lambda_e$ are right-constant so is
$\gamma_e$, and hence, there is an $\alpha_2 > 0$ such that \eqref{eq:alpha_capacity:timedep},
\eqref{eq:alpha_inrate:timedep} and \eqref{eq:alpha_speed_labels:timedep} are fulfilled for all $\xi \in [0, \alpha_2)$.
Clearly, $\alpha \coloneqq \min \Set{\alpha_1, \alpha_2} > 0$ is feasible.
\end{proof}

For the maximal feasible $\alpha$ we call the interval $[\phi, \phi + \alpha)$ a \emph{thin flow phase}.
\begin{restatable}{lemma}{extensionisflowtimedep} \label{lem:extension_is_flow:timedep}
An $\alpha$-extension is a feasible flow over time and the extended $\l$-labels coincide with the earliest arrival times,
i.e., they satisfy \Cref{eq:l_v_def:timedep} for all $\varphi \in [\phi, \phi + \alpha)$.
\end{restatable}
%The technical proof of this lemma can be found in the appendix on page \pageref{proof:extension_is_flow:timedep}.

The final step is to show that an $\alpha$-extension is a restricted Nash flow over time on $[0, \phi + \alpha)$ and
that we can continue this process up to $\infty$.

\begin{restatable}{theorem}{extension} \label{thm:extension:timedep}
Given a restricted Nash flow over time $f = (f_e^+, f_e^-)_{e\in E}$ on $[0, \phi)$ in a time-varying network and a
feasible $\alpha > 0$ then the $\alpha$-extension is a restricted Nash flow over time on~$[0, \phi + \alpha)$.
\end{restatable}

\begin{proof}\label{proof:extension:timedep}
Lemma~\ref{lem:nash_flow_characterization:timedep} yields $F_e^+(\l_u(\varphi)) = F_e^-(\l_v(\varphi))$ for all~$\varphi
\in [0, \phi)$, so for $\xi \in [0, \alpha)$ it holds that
\[F_e^+(\l_u(\phi + \xi)) = F_e^+(\l_u(\phi)) + \frac{x'_e}{\l'_u} \cdot \xi \cdot \l'_u = 
 F_e^-(\l_v(\phi)) + \frac{x'_e}{\l'_v} \cdot \xi \cdot \l'_v = F_e^-(\l_v(\phi + \xi)).\]
It follows again by \Cref{lem:nash_flow_characterization:timedep} together with \Cref{lem:extension_is_flow:timedep}
that the $\alpha$-extension is a restricted Nash flow over time on~$[0, \phi + \alpha)$.
\end{proof}

Finally, we obtain our main result:
\begin{restatable}{theorem}{finishingnashflow} \label{thm:finishing_nash_flow:timedep}
There exists a Nash flow over time in every time-varying network with right-constant speed limits, capacities and
network inflow rates.
\end{restatable}
\begin{proof}\label{proof:finishing_nash_flows:timedep}
The process starts with the empty flow over time, i.e., a restricted Nash flow over time for~$[0, 0)$. We apply
\Cref{thm:extension:timedep} with a maximal feasible $\alpha$. If one of the $\alpha$ is unbounded we are done.
Otherwise, we obtain a sequence $(f_i)_{i\in\N}$, where $f_i$ is a restricted Nash flow over time for $[0,\phi_i)$, with
a strictly increasing sequence $(\phi_i)_{i \in \N}$. In the case that this sequence has a finite limit, say
$\phi_\infty < \infty$, we define a restricted Nash flow over time $f^\infty$ for $[0, \phi_\infty)$ by using the
point-wise limit of the $x$- and $\l$-labels, which exists due to monotonicity and boundedness of these functions. Note
that there are only finitely many different thin flows, and therefore, the derivatives $x'$ and $\l'$ are bounded. Then
the process can be restarted from this limit point. This so called \emph{transfinite induction} argument works as
follows: Let $\mathcal{P}_G$ be the set of all particles $\phi \in \Flow$ for which there exists a restricted Nash flow
over time on $[0, \phi)$ constructed as described above. The set $\mathcal{P}_G$ cannot have a maximal element because
the corresponding Nash flow over time could be extended by using \Cref{thm:extension:timedep}. But $\mathcal{P}_G$
cannot have an upper bound either since the limit of any convergent sequence would be contained in this set. Therefore,
there exists an unbounded increasing sequence $(\phi_i)_{i = 1}^\infty \in \mathcal{P}_G$. As a restricted Nash flow
over time on $[0, \phi_{i+1}]$ contains a restricted Nash flow over time on $[0, \phi_i]$ we can assume that there
exists a sequence of \emph{nested} restricted Nash flow over time. Hence, we can construct a Nash flow over time $f$ on
$[0, \infty)$ by taking the point-wise limit of the $x$- and $\l$-labels, completing the proof.
\end{proof}

\begin{figure}[p]
\centering \includegraphics[page=1]{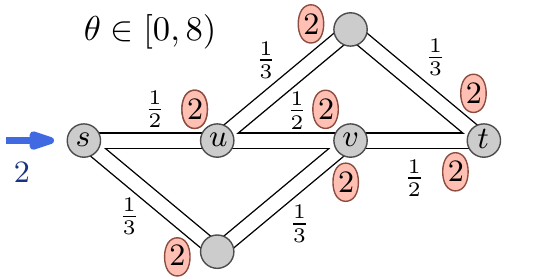}\includegraphics[page=2]{nash_flow_timedep} 
\\[0.5em]{\color{gray}\hrule}\vspace{0.5em}
\includegraphics[page=1]{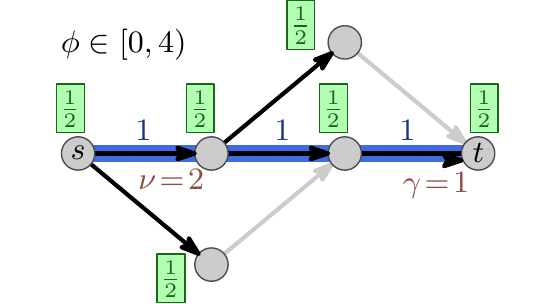}\includegraphics[page=2]{thin_flows_timedep}
\includegraphics[page=3]{thin_flows_timedep}\includegraphics[page=4]{thin_flows_timedep}
\includegraphics[page=5]{thin_flows_timedep}\includegraphics[page=6]{thin_flows_timedep}\\[0.5em]
{\color{gray}\hrule}\vspace{0.5em}
\includegraphics[page=3]{nash_flow_timedep}\includegraphics[page=4]{nash_flow_timedep}\\[-0.5em]
\includegraphics[page=5]{nash_flow_timedep}\includegraphics[page=6]{nash_flow_timedep} 
\caption{A Nash flow over time with seven thin flow phases in a time-varying network. } \label{fig:nash_flow_example:timedep}
\end{figure}

\subsubsection{Example.} An example of a Nash flow over time in a time-varying network together with the corresponding
thin flows is shown in \Cref{fig:nash_flow_example:timedep} on the next page. \emph{On the top:} The network properties
before time $8$ (\emph{left side}) and after time $8$ (\emph{right side}). \emph{In the middle:} There are seven thin
flow phases. Note that the third and forth phase (both depicted in the same network) are almost identical and only the
speed ratio of arc $v \sink$ changes, which does not influence the thin flow at all. \emph{At the bottom:} Some key
snapshots in time of the resulting Nash flow over time. The current speed limit $\lambda_{v \sink}$ is visualized by the
length of the green arrow and, for $\theta \geq 8$, the reduced capacity $\capa_{\source u}(\theta)$ is displayed by a
red bottle-neck.

As displayed at the top the capacity of arc $\source u$ drops from $2$ to $1$ at time $8$ and, at the same time, the
speed limit of arc $v\sink$ decreases from $\frac{1}{2}$ to $\frac{1}{6}$. The first event for particle $4$ is due to a
change of the speed ratio leading to an increase of $\l'_\sink$. For particle $6$, the top path becomes active and is
taken by all following flow as particles on arc $v \sink$ are still slowed down. For particle $8$, the speed ratio at
arc $v \sink$ changes back to~$1$ but, as this arc is inactive, this does not change anything. Particle $12$ is the
first to experience the reduced capacity on arc $\source u$. The corresponding queue of this arc increases until the
bottom path becomes active. This happens in two steps: first only the path up to node $v$ becomes active for $\phi =
16$, and finally, the complete path is active from $\phi = 20$ onwards.

\section{Conclusion and Open Problems} \label{sec:conclusion} In this paper, we extended the base model that was
introduced by Koch and Skutella, to networks which capacities and speed limits that changes over time. We showed that
all central results, namely the existence of dynamic equilibria and their underlying structures in form of thin flow
with resetting, can be transfered to this new model. With these new insights it is possible to model more general
traffic scenarios in which the network properties are time-dependent. In particular, the flooding evacuation scenario,
which was mentioned in the introduction, could not be modeled (not even approximately) in the base model.

There are still a lot of open question concerning time-varying networks. For example, it would be interesting to
consider other flows over time in this setting, such as earliest arrival flows or instantaneous dynamic equilibria (see
\cite{graf2020ide}) and show their existence. Can the proof for the bound of the price of anarchy
\cite{correa2019price} be transfered to this model, or is it possible to construct an example where the price of anarchy
is unbounded? 

\bibliographystyle{splncs04} \bibliography{literature}

%%%%%%%%%%%%%%%%%%%%%%%%%%%%%%%%%%%%%%%%%%%%%%%%%%%%%%%%%%%%%%%%%%%%%%
%%%%%%%%%%%%%% Appendix: Time-Dependent %%%%%%%%%%%%%%%%%%%%%%%%%%%%%%
%%%%%%%%%%%%%%%%%%%%%%%%%%%%%%%%%%%%%%%%%%%%%%%%%%%%%%%%%%%%%%%%%%%%%%
\ifarxiv
\newpage
\section{Appendix: Technical Proofs} \label{sec:appendix}

\transittimestimedep*
\begin{proof} \label{proof:transit_times:timedep} \phantom{a}
\begin{enumerate}
\item[\ref{it:queueing_time_monotone:timedep}] Consider two points in time $\theta_1 < \theta_2$, then $\tau \coloneqq
\theta_1 - \theta_2 + \tau_e(\theta_1)$ is strictly smaller than $\tau_e(\theta_2)$ since
\[\int_{\theta_2}^{\theta_2 + \tau}\lambda_e(\xi) \diff \xi =  \int_{\theta_2}^{\theta_1 + \tau_e(\theta_1)} \lambda_e(\xi)
\diff \xi < \int_{\theta_1}^{\theta_1 + \tau_e(\theta_1)} \lambda_e(\xi) \diff \xi = 1,\]
where the strict inequality holds, since $\lambda_e$ is always strictly positive. The last equality follows by the
definition of $\tau_e(\theta_1)$. Hence, with the definition of $\tau_e(\theta_2)$ we have $\theta_1 + \tau_e(\theta_1)
< \theta_2 + \tau_e(\theta_2)$.

\item[\ref{it:tau_is_diffbar:timedep}] Since $\theta \mapsto \theta + \tau_e(\theta)$ is monotone, Lebesgue's theorem
for the differentiability of monotone functions implies that it is almost everywhere differentiable. The same is then
true for $\tau_e$. The continuity follows directly from the definition since $\lambda_e$ is always strictly positive.
 
\item[\ref{it:derivative_of_tau:timedep}]
By the definition of $\tau_e(\theta)$ we have
\[\int_0^{\theta + \tau_e(\theta)} \lambda_e(\xi) \diff \xi - \int_0^{\theta} \lambda_e(\xi) \diff \xi = 1.\]
Taking the derivatives of both sides and using Lebesgue's differentiation theorem together with the chain rule, we obtain
\[\lambda_e(\theta + \tau_e(\theta)) \cdot (1 + \tau'_e(\theta)) - \lambda_e(\theta) = 0.\]
Since $\lambda_e$ is always strictly positive, we get
\[1 + \tau'_e(\theta) = \frac{\lambda_e(\theta)}{\lambda_e(\theta + \tau_e(\theta))}.\]\vspace{-2em}
\end{enumerate}
\end{proof}

\technicalpropertiestimedep*
\begin{proof}\label{proof:technical_properties:timedep}\phantom{a}
\begin{enumerate}
\item[\ref{it:q_equiv_z:timedep}] This follows directly from the definition of the waiting time $q_e$.

\item[\ref{it:positive_queue_while_emptying:timedep}] By equation \eqref{eq:outflow:timedep} we have that $f_e^-(\xi) \leq
\capa_e(\xi)$ almost everywhere. Hence, we have by definition that $q_e(\theta)$ is the minimal value such that
\[\int_{\theta + \tau_e(\theta)}^{\theta + \tau_e(\theta) + q_e(\theta)} \capa_e(\xi) \diff \xi 
= z_e(\theta + \tau_e(\theta)).\]
Thus, we obtain for $\xi \in [0, q_e(\theta))$ that
\begin{align*}
F_e^-(\theta + \tau_e(\theta) + \xi) - F_e^-(\theta + \tau_e(\theta)) 
&= \int_{\theta + \tau_e(\theta)}^{\theta + \tau_e(\theta) + \xi} f_e^-(\xi) \diff \xi \\
&\leq \int_{\theta + \tau_e(\theta)}^{\theta + \tau_e(\theta) + \xi} \capa_e(\xi) \diff \xi\\
&< z_e(\theta + \tau_e(\theta)) \\
&= F_e^+(\theta) - F_e^-(\theta + \tau_e(\theta)).
\end{align*}
Or in short: $F_e^+(\theta) - F^-_e(\theta+\tau_e(\theta) + \xi) > 0$ for $\xi \in [0,q_e(\theta))$. Since $F_e^+$ is
non-decreasing we obtain for all $\xi \in [0,q_e(\theta))$ that
\[z_e(\theta + \tau_e(\theta) + \xi)\!=\!F_e^+(\theta+\xi)-F^-_e(\theta+\tau_e(\theta) + \xi) 
\!\geq\! F_e^+(\theta)-F^-_e(\theta+\tau_e(\theta) + \xi) \!>\! 0.\]

\item[\ref{it:in_equals_out_at_exit_time:timedep}] By \eqref{eq:outflow:timedep} and
\ref{it:positive_queue_while_emptying:timedep} we obtain for almost all ${\xi \in [\theta + \tau_e(\theta), \theta +
\tau_e(\theta) + q_e(\theta))}$ that $f_e^-(\xi) = \capa_e(\xi)$. By the definition of $q_e$ we have
\begin{align*}
F_e^-(\theta + \tau_e(\theta) + q_e(\theta)) - F_e^-(\theta + \tau_e(\theta))
&= \int_{\theta + \tau_e(\theta)}^{\theta + \tau_e(\theta) + q_e(\theta)} f_e^-(\xi) \diff \xi\\
&= \int_{\theta + \tau_e(\theta)}^{\theta + \tau_e(\theta) + q_e(\theta)} \capa_e(\xi) \diff \xi \\
&= z_e(\theta + \tau_e(\theta)) \\
&= F_e^+(\theta) - F_e^-(\theta + \tau_e(\theta)).
\end{align*}
Hence, $F_e^-(T_e(\theta)) = F_e^+(\theta)$.

\item[\ref{it:equal_exit_times:timedep}]
Since $F_e^+(\theta_1) = F_e^+(\theta_2)$ we obtain with the monotonicity of
$F_e^-$ together with \Cref{lem:transit_times:timedep}~\ref{it:queueing_time_monotone:timedep} that
\begin{align*}z_e(\xi + \tau_e(\xi)) &= F_e^+(\xi) - F_e^-(\xi + \tau_e(\xi)) \\
&\geq F_e^+(\theta_2) - F_e^-(\theta_2 + \tau_e(\theta_2)) = z_e(\theta_2 + \tau_e(\theta_2)) > 0,\end{align*}
hence, \eqref{eq:outflow:timedep} provides $f_e^-(\xi) = \capa_e(\xi)$ for almost all ${\xi \in [\theta_1 +
\tau_e(\theta_1), \theta_2 + \tau_e(\theta_2)]}$.

Thus, the definition of $q_e$ implies that $q_e(\theta_1)$ equals
\begin{align*}
&\min \Set{ q \geq 0 | \begin{array}{l}
\displaystyle \int_{\theta_1 + \tau_e(\theta_1)}^{\theta_2 + \tau_e(\theta_2)} f_e^-(\xi) \diff\xi 
+ \int_{\theta_2 + \tau_e(\theta_2)}^{\theta_1 + \tau_e(\theta_1) +q} \capa_e(\xi) \diff\xi \qquad \\[1.5em] 
\hfill \qquad \qquad
= F_e^+(\theta_1) -F_e^-(\theta_1 + \tau_e(\theta_1))
\end{array} }\\
=&\min \Set{ p \geq 0 | \int_{\theta_2 + \tau_e(\theta_2)}^{\theta_2 + \tau_e(\theta_2) +p} \!\!\!\!\capa_e(\xi) \diff\xi 
= F_e^+(\theta_2) \!- \! F_e^-(\theta_2 \!+\! \tau_e(\theta_2))} \\
&\hspace{6.5cm} + \theta_2 + \tau_e(\theta_2) 
\!-\! \theta_1 \!-\! \tau_e(\theta_1)\\
=&\;q_e(\theta_2) + \theta_2 + \tau_e(\theta_2) - \theta_1 - \tau_e(\theta_1).
\end{align*}
Here, we substitute $q$ by ${p + \theta_2 + \tau_e(\theta_2) - \theta_1 - \tau_e(\theta_1)}$ in order to obtain the
first equation. Note that the condition $p \geq 0$ is always satisfied since the right hand side $F_e^+(\theta_2)
-F_e^-(\theta_2 + \tau_e(\theta_2))$ equals $z_e(\theta_2 + \tau_e(\theta_2))$ and is therefore strictly positive by
assumption. Hence, we obtain
\[T_e(\theta_1) = \theta_1 + \tau_e(\theta_1) +q_e(\theta_1)=\theta_2 + \tau_e(\theta_2) +q_e(\theta_2) = T_e(\theta_2).\]

\item[\ref{it:T_monoton:timedep}] Considering two points in time $\theta_1 < \theta_2$, we show that $T_e(\theta_1) \leq
T_e(\theta_2)$. Since $F_e^+$ is non-decreasing, \ref{it:in_equals_out_at_exit_time:timedep} implies that
\begin{equation} \label{eq:outflow_with_T_monoton:timedep}
F_e^-(T_e(\theta_2))=F_e^+(\theta_2)\geq F_e^+(\theta_1) = F_e^-(T_e(\theta_1)). 
\end{equation}
If this holds with strict inequality, we obtain by monotonicity of $F_e^-$ that $T_e(\theta_1) < T_e(\theta_2)$. If
equation \eqref{eq:outflow_with_T_monoton:timedep} holds with equality we have two cases. If ${z_e(\theta_2 +
\tau_e(\theta_2))>0}$, \ref{it:equal_exit_times:timedep} states that $T_e(\theta_1)=T_e(\theta_2)$. If~$z_e(\theta_2 +
\tau_e(\theta_2))=0$, \ref{it:positive_queue_while_emptying:timedep} applied to $\theta_1$ implies that $\xi \coloneqq
\theta_2 + \tau_e(\theta_2)-\theta_1 - \tau_e(\theta_1) \not\in [0, q_e(\theta_1))$. Since $\xi \geq 0$ by
\Cref{lem:transit_times:timedep}~\ref{it:queueing_time_monotone:timedep} we have $\xi \geq q_e(\theta_1)$, and thus,
\[T_e(\theta_2) \stackrel{\text{\ref{it:q_equiv_z:timedep}}}{=} \theta_2 + \tau_e(\theta_2) \geq \theta_1 + \tau_e(\theta_1)
 + q_e(\theta_1) = T_e(\theta_1).\]

\item[\ref{it:q_is_diffbar:timedep}] The continuity of $q_e$ follows since $\capa_e$ is always strictly positive and
$z_e$ is continuous, as it is the difference of two continuous functions. Finally, $T_e$ is continuous since it is the
sum of three continuous functions.

By \ref{it:T_monoton:timedep} the function $T_e$ is non-decreasing for all $e \in E$, and hence, Lebesgue's theorem for
the differentiability of monotone functions states that $T_e$ is almost everywhere differentiable. Since ${\theta
\mapsto \theta + \tau_e(\theta)}$ is monotone this also holds for $\tau_e$ since it is the difference of two almost
everywhere differentiable functions. As a sum of almost everywhere differentiable functions, $q_e(\theta) = T_e(\theta)
- \tau_e(\theta) - \theta$ has this property as well.

\item[\ref{it:derivative_of_T:timedep}] The definition of $q_e(\theta)$ states that
\[\int_0^{T_e(\theta)} \capa_e(\xi) \diff \xi - \int_0^{\theta + \tau_e(\theta)} \capa_e(\xi) \diff \xi 
= z_e(\theta + \tau_e(\theta)) = F_e^+(\theta) - F_e^-(\theta + \tau_e(\theta)).\]

Taking the derivative on both sides we obtain by using the chain rule that

\[\capa_e(T_e(\theta)) \cdot T'_e(\theta) - \capa_e(\theta + \tau_e(\theta)) \cdot (1 + \tau'_e(\theta)) 
= f_e^+(\theta) - f_e^-(\theta + \tau_e(\theta)) \cdot (1 + \tau'_e(\theta)).\]

If $q_e(\theta) > 0$ we have by equation \eqref{eq:outflow:timedep} that $f_e^-(\theta + \tau_e(\theta)) = \capa_e(\theta +
\tau_e(\theta))$, and therefore, dividing by $\capa_e(T_e(\theta))$ (which is strictly positive by assumption) yields
\[T'_e(\theta) = \frac{f_e^+(\theta)}{\capa_e(T_e(\theta))}.\]

For $q_e(\theta) = 0$ we have $f_e^-(\theta + \tau_e(\theta)) = \min \Set{\frac{f_e^+(\theta)}{\gamma_e(\theta)},
\capa_e(\theta + \tau_e(\theta))}$ and $T_e(\theta) = \theta + \tau_e(\theta)$. Hence, dividing by $\capa_e(\theta +
\tau_e(\theta)) = \capa_e(T_e(\theta))$ and using \Cref{lem:transit_times:timedep}.\ref{it:derivative_of_tau:timedep}
provides
\begin{align*}T'_e(\theta) &= \gamma_e(\theta) + \frac{f_e^+(\theta)}{\capa_e(T_e(\theta))} - 
\min \Set{\frac{f_e^+(\theta)}{\gamma_e(\theta)}, \capa_e(T_e(\theta))}\cdot \frac{\gamma_e(\theta)}{\capa_e(T_e(\theta))}\\
&= \max \Set{\gamma_e(\theta), \frac{f_e^+(\theta)}{\capa_e(T_e(\theta))}},
\end{align*}
which finishes the proof.
\end{enumerate}
\end{proof}

\existenceofthinflowstimedep*
\begin{proof}\label{proof:existence_of_thin_flows:base}
Let $\X$ be the compact, convex and non-empty set of all static $\source$-$\sink$-flows of value $1$ and let
$\Gamma \colon \X \to 2^\X$ be defined by
\[x' \mapsto \Set{y' \in \X | y'_e = 0 \text{ for all } e = uv \in E' \text{ with } \l'_v < \rho_e(\l'_u, x'_e)},\]

where $(\l'_v)_{v \in V}$ are the node labels associated with $x'$ uniquely defined by
\begin{equation}\label{eq:l'_labels:base}
 \l'_v =  \begin{cases}\quad\frac{1}{r} & \text{ if } v = \source, \\ 
  \min\limits_{e = uv\in E'} \rho_e(\l'_u, x'_e) &\text{ if } v
 \in V\backslash \set{\source}.
\end{cases}
\end{equation}

The existence of a fixed point of $\Gamma$ is provided by Kakutani's fixed point theorem \cite{kakutani1941generalization}
\begin{theorem}[Kakutani's Fixed Point Theorem] \label{thm:kakutani}
Let $\X$ be a compact, convex and non-empty subset of $\R^n$ and $\Gamma\colon \X \to 2^\X$, such that for every
$x \in \X$ the image $\Gamma(x)$ is non-empty and convex. Suppose the set $\Set{(x, y) | y \in \Gamma(x)}$ is
closed. Then there exists a fixed point $x^* \in \X$ of $\Gamma$, i.e.,~$x^* \in \Gamma(x^*)$.
\end{theorem}
All conditions are satisfied:
\begin{itemize}
\item The set $\Gamma(x')$ is non-empty, because there has to be at least one path $P$ from $\source$ to $\sink$ with
$\l'_v = \rho_e(\l'_u, x'_e)$ for each arc $e$ on $P$. If we set $y_e = 1$ for all arcs $e$ on $P$ and set every other
value to~$0$ we obtain an element in~$\Gamma(x')$.

\item To see that $\Gamma(x')$ is convex, note that the arcs that can be used for sending flow, i.e., the ones
satisfying $\l'_v = \rho_e(\l'_u, x'_e)$, are fixed within the set $\Gamma(x')$. Furthermore, every convex combination
$y$ of two elements $y^1, y^2 \in \Gamma(x')$ only uses arcs that are also used by $y^1$ or $y^2$.

\item In order to show that the function graph $\Set{(x', y') | y' \in \Gamma(x)}$ is closed let $(x^n, y^n)_{n\in \N}$
be a sequence within this set, i.e.,~$y^n \in \Gamma(x^n)$. Since both sequences, $(x^n)_{n\in \N}$ and $(y^n)_{n\in
\N}$, are contained in the compact set $\X$ they both have a limit $x^*$ and $y^*$ within~$\X$. Let $(\l^n)_{n \in \N}$ be
the sequence of associated node labels of $(x^n)$ and $\l^*$ the node label of~$x^*$. Note that the mapping $x' \mapsto
\l'$ is continuous, and therefore, it holds that~$\l^* = \lim_{n \to \infty} \l^n$. We prove that~$y^* \in \Gamma(x^*)$.
Suppose there is an arc $e = uv \in E'$ with $y^*_e > 0$ and~$\l^*_v < \rho_e(\l^*_u, x^*_e)$. But since $\rho_e$ is
continuous there has to be an $n_0 \in \N$ such that $y^n_e > 0$ and $\l^n_v < \rho_e(\l^n_u, x^n_e)$ for all~$n \geq
n_0$, which is a contradiction. Hence, $\Set{(x', y') | y' \in \Gamma(x)}$ is closed.
\end{itemize}

Since all conditions for Kakutani's fixed point theorem are satisfied, there has to be a fixed point~$x^*$ of~$\Gamma$.
Let $\l^*$ be the corresponding node labeling. We show that the pair $(x^*,\l^*)$ satisfies the thin flow conditions.
\Cref{eq:l'_s:timedep,eq:l'_v_min:timedep} follow immediately by \eqref{eq:l'_labels:base}. For every arc $e = uv \in
E'$ with $x^*_e > 0$ it holds that $\l^*_v = \rho_e(\l^*_u, x^*_e)$ since $x^* \in \Gamma(x^*)$, which shows
\Cref{eq:l'_v_tight:timedep}. Thus, $(x^*,\l^*)$ forms a thin flow with resetting, which completes the proof.
\end{proof}

\extensionisflowtimedep*
\begin{proof} \label{proof:extension_is_flow:timedep}
Flow conservation on nodes \eqref{eq:flow_conservation:timedep} holds since for all ${\theta \in [\l_v(\phi),
\l_v(\phi + \alpha))}$ we have
\[\sum_{e\in \delta^+(v)} \!\!\!f_e^+(\theta) - \!\!\!\sum_{e \in \delta^-(v)}\!\!\! f_e^-(\theta) 
= \!\!\!\sum_{e\in \delta^+(v)} \frac{x'_e}{\l'_v} -\!\!\! \sum_{e \in \delta^-(v)} \frac{x'_e}{\l'_v} 
= \begin{cases}
0 & \hspace{-1cm}\text{if } v \in V \setminus \set{\source, \sink} \\
r(\l_\source(\phi)) \stackrel{\text{\eqref{eq:alpha_inrate:timedep}}}{=} \theta  & \text{if } v = \source.
\end{cases}\]
Next, we show that the feasibility condition \eqref{eq:outflow:timedep} is satisfied. For this we first consider arcs
$e$ with $x'_e > 0$, which implies $e \in E'_\phi$. By \eqref{eq:l'_v_tight:timedep} we have that $\l'_v \geq
\gamma_e(\l_u(\phi)) \cdot \l'_u$. Since $\gamma$ is constant during the thin flow phase, so is $\tau'$, and therefore,
we have for all $\xi \in [0, \alpha)$ that
\begin{align*}
\l_v(\phi + \xi) &= \l_v(\phi) + \xi \cdot \l'_v\\
&\geq \l_v(\phi) + \xi \cdot \gamma_e(\l_u(\phi)) \cdot \l'_u\\
&\geq \l_u(\phi) + \tau_e(\l_u(\phi)) + \xi \cdot (1 + \tau'_e(\l_u(\phi)) \cdot \l'_u \\
&= \l_u(\phi + \xi) 
+ \tau_e(\l_u(\phi + \xi)).
\end{align*}
In other words, $e$ stays active during the thin flow phase.

We consider the outflow rate at time $\theta + \tau_e(\theta)$ for $\theta \in [\l_u(\phi), \l_u(\phi + \alpha))$. In
the case of $\theta + \tau_e(\theta) < \l_v(\phi)$ the feasibility condition follows from prior phases. Otherwise,
$\theta + \tau_e(\theta) \in [\l_v(\phi), \l_v(\phi + \alpha))$, and therefore,
\begin{align*}f_e^-(\theta + \tau_e(\theta)) &= \frac{x'_e}{\l'_v} \stackrel{\text{\eqref{eq:l'_v_tight:timedep}}}{=}
\frac{x'_e}{\rho_e(\l'_u, x'_e)} \\
&= \begin{cases}
\min\Set{\frac{x'_e}{\gamma_e(\l_u(\phi)) \cdot \l'_u}, \capa_e(\l_v(\phi))}& \text{ if } e \in E'_\phi \setminus E^*_\phi,\\
\capa_e(\l_v(\phi))& \text{ else,}
\end{cases}\\
&= \begin{cases}
\min\Set{\frac{f_e^+(\theta)}{\gamma_e(\theta)}, \capa_e(\theta + \tau_e(\theta))}& \text{ if } q_e(\theta) = 0,\\
\capa_e(\theta + \tau_e(\theta))& \text{ else.}
\end{cases}\end{align*}
In the case that $x'_e = 0$ we either have $\l'_v = 0$, but then there is nothing to show since the interval
$[\l_v(\phi), \l_v(\phi + \alpha))$ would be empty, or $\l'_v > 0$, which means by \eqref{eq:l'_v_min:timedep} that
either $e$ is not active, or it is active but non-resetting. In both cases we have $q_e(\l_u(\theta)) = 0$ and since
$f^+_e(\l_u(\theta)) = 0$ for all $\theta \in [\l_u(\phi, \l_u(\phi + \alpha))$ the queue stays empty during this phase.
\eqref{eq:outflow:timedep} follows since ${f_e^-(\theta + \tau_e(\theta)) = \frac{x'_e}{\l'_v} = 0 = f^+_e(\theta)}$ holds
for all $\theta \in [\l_u(\phi, \l_u(\phi + \alpha))$.
Altogether, we showed that the $\alpha$-extension is indeed a feasible flow over time.

\smallskip It remains to show that \Cref{eq:l_v_def:timedep} holds, which implies that the extended $\l$-functions denote the
earliest arrival times. First of all we have
\[\int_0^{\l_\source(\phi + \xi)} \inrate(\xi) \diff \xi 
= \phi + \int_{\l_\source(\phi)}^{\l_\source(\phi + \xi)} \inrate(\xi) \diff \xi 
= \phi + \inrate(\l_\source(\phi)) \cdot \l'_\source \cdot \xi = \phi + \xi.\]
Since $\inrate$ is always strictly positive, $\l_\source(\phi)$ is the minimal value with this property, which shows
\eqref{eq:l_v_def:timedep} for $v = \source$. For $v \neq \source$ we distinguish between three cases for every given
arc $e = uv \in E$.

\paragraph{Case 1:}~$e \in E\backslash E'_\phi$.\\
Since $\alpha$ satisfies \eqref{eq:alpha_others:timedep} it is satisfied for all $\xi \in [0, \alpha)$, and hence,
\begin{align*}\l_v(\phi + \xi) = \l_v(\phi) + \xi \cdot \l'_v &\!\stackrel{\eqref{eq:alpha_others:timedep}}{\leq}\!
 \l_u(\phi) + \xi \cdot \l'_u + \tau_e(\l_u(\phi) + \xi \cdot \l'_u)\\
&\leq \l_u(\phi + \xi) + 
 \tau_e(\l_u(\phi + \xi)) \leq \T_e(\l_u(\phi + \xi)).\end{align*}

\paragraph{Case 2:} $e \in E'_\phi \backslash E^*_\phi$ and~$\gamma_e(\l_u(\phi)) \cdot \l'_u \geq
\frac{x'_e}{\capa_e(\l_v(\phi))}$.\\ Since $e$ is active we have $\l_v(\phi) = \T_e(\l_u(\phi)) = \l_u(\phi) +
\tau_e(\l_u(\phi))$ and \eqref{eq:l'_v_min:timedep} implies $\l'_v \leq \gamma_e(\l_u(\phi)) \cdot \l'_u$. No
queue builds up as ${f_e^+(\l_u(\phi+\xi)) = \frac{x'_e}{\l'_u} \leq \capa_e(\l_v(\phi))}$, which means
$z_e(\l_u(\phi + \xi)+\tau_e(\l_u(\phi))) = 0$ for all $\xi \in (0, \alpha]$. Combining these yields
\begin{align*}
\l_v(\phi + \xi) & \stackrel{\eqref{eq:l'_v_min:timedep}}{\leq} \l_v(\phi) + \xi \cdot \gamma_e(\l_u(\phi)) \cdot \l'_u\\
 &\;\;\,=\;\; \l_u(\phi) + \tau_e(\l_u(\phi)) + \xi \cdot (1 + \tau'_e(\l_u(\phi)) \cdot \l'_u\\
 &\;\;\,=\;\; \l_u(\phi + \xi) + \tau_e(\l_u(\phi + \xi))\\
 &\;\;\,=\;\; T_e(\l_u(\phi + \xi)).
\end{align*}

\paragraph{Case 3:}  $e \in E^*_\phi$ or $\left(e \in E'_\phi \text{ and } \gamma_e(\l_u(\phi)) \cdot \l'_u <
\frac{x'_e}{\capa_e(\l_v(\phi))}\right)$.\\ Arc $e$ is active ,i.e., $\l_v(\phi) = \T_e(\l_u(\phi))$. We have
$\rho_e(\l'_u, x'_e) = \frac{x'_e}{\capa_e(\l_v(\phi))}$, and hence, \eqref{eq:l'_v_min:timedep} implies~$\l'_v \leq
\frac{x'_e}{\capa_e(\l_v(\phi))}$. \Cref{lem:technical_properties:timedep}~\ref{it:derivative_of_T:timedep} yields
\[T'_e(\l_u(\phi))= \frac{f_e^+(\l_u(\phi))}{\capa_e(\l_v(\phi))} = \frac{x'_e}{\l'_u \cdot \capa_e(\l_v(\phi))}\] 

since either $q_e(\l_u(\phi)) > 0$ (if $e \in E^*$) or, in the case of $e \notin E^*_\phi$, we have

\[\frac{f_e^+(\l_u(\phi))}{\capa_e(\l_v(\phi))} = \frac{x'_e}{\l'_u \cdot \capa_e(\l_v(\phi))} > \gamma_e(\l_u(\phi)).\]

Hence, for all $\xi \in
\:(0, \alpha]$ we obtain
\begin{align*}\l_v(\phi + \xi) \!\stackrel{\text{\eqref{eq:l'_v_min:timedep}}}{=}\! \l_v(\phi) + \xi \cdot \l'_v
&\leq \l_v(\phi) + \xi \cdot \frac{x'_e}{\capa_e}  \\
&= T_e(\l_u(\phi)) + \xi \cdot T'_e(\l_u(\phi)) \cdot \l'_v
= \T_e(\l_u(\phi + \xi)).\end{align*}
This shows that there is no arc with an exit time earlier than the earliest arrival time, and therefore, the left hand
side of \eqref{eq:l_v_def:timedep} is always smaller or equal to the right hand side.

It remains to show that the equation holds with equality. For every node $v \in V \setminus \set{\source}$ there is at
least one arc $e \in E'$ with $\l'_v = \rho_e(\l'_u, x'_e)$ in the thin flow due to \eqref{eq:l'_v_min:timedep}. No
matter if this arc belongs to Case 2 or Case 3 the corresponding equation holds with equality, which shows for all $\xi
\in (0, \alpha]$ that
\[\l_v(\phi + \xi) = \min_{e = uv \in E} \T_e(\l_u(\phi + \xi)).\]
This shows that \eqref{eq:l_v_def:timedep} is also satisfied for $v \neq \source$, which completes the proof.
\end{proof}

\begin{lemma}[Differentiation rule for a minimum]\label{lem:diff_rule_for_min:pre}
For every element $e$ of a finite set $E$ let $T_e \colon \R_{\geq 0} \rightarrow \R$ be a function that is
differentiable almost everywhere and let $\l(\theta) \coloneqq \min_{e \in E} T_e(\theta)$ for all $\theta \geq 0$. It holds
that $l$ is almost everywhere differentiable with
\begin{equation} \label{eq:diff_rule_for_min:pre}
\l'(\theta) = \min_{e \in E'_\theta} T'_e(\theta)
\end{equation}
for almost all $\theta \geq 0$ where $E'_\theta \coloneqq \set{e \in E| \l(\theta) = T_e(\theta)}$.
\end{lemma}

\begin{proof}
Let $\phi \geq 0$ such that all $T_e$, for all $e \in E$, are differentiable, which is almost everywhere.
Since all functions $T_e$ are continuous at $\phi$ we have for sufficiently small $\epsilon > 0$ that $\l(\phi +
\xi) = \min_{e \in E'_\phi} T_e(\phi + \xi)$ for all $\xi \in [\phi, \phi + \epsilon]$. It follows that

\begin{align*}\lim_{\xi \,\searrow\, 0} \frac{\l(\phi + \xi) - \l(\phi)}{\xi}
&= \lim_{\xi \,\searrow\, 0}  \min_{e \in E'_\phi} \frac{T_e(\phi + \xi) - \l(\phi)}{\xi}\\
&= \min_{e \in E'_\phi} \lim_{\xi \,\searrow\, 0} \frac{T_e(\phi + \xi) - T_e(\phi)}{\xi}
= \min_{e \in E'_\phi} T'_e(\phi).\end{align*}
Note that every point $\phi$ where all $T_e$ are differentiable, but for which the left derivative of $\l$ does not
coincide with the right derivative of $\l$, is a proper crossing of at least two $T_e$ functions. Therefore, these
points are isolated and form a null set. Hence, we have $\l'(\phi) = \min_{e \in E'_\phi} T'_e(\phi)$ for almost
all~$\phi \in \R_{\geq 0}$.
\end{proof}
\fi

\end{document}